\documentclass{amsart}
\usepackage{amsmath}
\usepackage{amsfonts}
\usepackage{amssymb}
\usepackage{amsthm}
\usepackage{graphicx}
\usepackage{enumitem}
\usepackage{verbatim}
\usepackage{tikz-cd}
\usepackage{float}
\usepackage{hyperref}

\theoremstyle{plain}% default
\newtheorem{theorem}{Theorem}[section]
\newtheorem{thm}[theorem]{Theorem}
\newtheorem{lemma}[theorem]{Lemma}
\newtheorem{proposition}[theorem]{Proposition}
\newtheorem{prop}[theorem]{Proposition}

\newtheorem{cor}[theorem]{Corollary}
\newtheorem{question}[theorem]{Question}

\newtheorem{rem}[theorem]{Remark}
\newtheorem*{mthm}{Theorem \ref{thm:main}}

\DeclareMathOperator{\Ram}{Ram}
\DeclareMathOperator{\vol}{vol}
\DeclareMathOperator{\Mat}{Mat}
\DeclareMathOperator{\SL}{SL}
\DeclareMathOperator{\PSL}{PSL}
\DeclareMathOperator{\Hy}{\mathbb{H}}
\DeclareMathOperator{\N}{\mathbb{N}}
\DeclareMathOperator{\Q}{\mathbb{Q}}
\DeclareMathOperator{\R}{\mathbb{R}}
\DeclareMathOperator{\Z}{\mathbb{Z}}
\DeclareMathOperator{\C}{\mathbb{C}}

\DeclareMathOperator{\Nr}{\mathrm{N}}
\DeclareMathOperator{\nrm}{nrm}
\DeclareMathOperator{\tr}{tr}
\DeclareMathOperator{\gen}{gen}
\newcommand{\Fr}[1]{\ensuremath{\mathfrak{#1}}}
\newcommand{\A}{\mathcal{A}}

\newcommand{\cO}{\mathcal{O}}
\newcommand{\bH}{\mathbb{H}}
\newcommand{\p}{\mathfrak{p}}
\newcommand{\q}{\mathfrak{q}}

\newcommand{\surf}[1]{\bH^2/#1}

\newcommand{\intring}[1]{\mathcal{O}_{#1}}
\newcommand\restr[2]{{% we make the whole thing an ordinary symbol
  \left.\kern-\nulldelimiterspace % automatically resize the bar with \right
  #1 % the function
  \right|_{#2} % this is the delimiter
  }}

\newcommand{\quat}[3]{\left(\frac{#1, #2}{#3}\right)}

\newlist{arrowlist}{itemize}{1}
\setlist[arrowlist]{label=$\Leftrightarrow$}

\title{On the genus of congruence surfaces from maximal orders}
\author{Eric Albers}
\address{Department of Mathematics\\ Temple University\\ Philadelphia, Pennsylvania 19122}
\email{eric.albers@temple.edu}
\author{Nicholas Miller}
\address{Department of Mathematics\\Indiana University\\Bloomington, IN 47405}
\email{nimimill@iu.edu}

\begin{document}

%~~~~~~~~~~~~~~~~~~~~~~~~~~~~~~~~~~~~~~~~~~~~~~~~~~~
%~~~~~~~~~~~~~~~~~~~~~~~~~~~~~~~~~~~~~~~~~~~~~~~~~~~
%~~~~~~~~~~~~~~~~~~~~~~~~~~~~~~~~~~~~~~~~~~~~~~~~~~~

\begin{abstract}
In this paper, we investigate a question of Breuillard and Reid concerning which genera can be obtained by closed congruence surfaces.
Specifically, we study a smaller set of objects, namely the closed congruence surfaces which can be constructed by a maximal order in a quaternion algebra, and show that there is no surface of genus $212$ in this class.
In particular, we show that Breuillard and Reid's question restricted to such surfaces has a negative answer.
\end{abstract}

%~~~~~~~~~~~~~~~~~~~~~~~~~~~~~~~~~~~~~~~~~~~~~~~~~~~
%~~~~~~~~~~~~~~~~~~~~~~~~~~~~~~~~~~~~~~~~~~~~~~~~~~~
%~~~~~~~~~~~~~~~~~~~~~~~~~~~~~~~~~~~~~~~~~~~~~~~~~~~

\maketitle

%~~~~~~~~~~~~~~~~~~~~~~~~~~~~~~~~~~~~~~~~~~~~~~~~~~~
%~~~~~~~~~~~~~~~~~~~~~~~~~~~~~~~~~~~~~~~~~~~~~~~~~~~
%~~~~~~~~~~~~~~~~~~~~~~~~~~~~~~~~~~~~~~~~~~~~~~~~~~~

\section{Introduction}

Given a surface $S$ of genus $g$ and $n$ punctures, a longstanding question in Teichm\"uller theory is to find and describe all hyperbolic metrics in Teichm\"uller space, $\mathcal{T}_{g,n}$, which maximize systolic length.
Specifically, for fixed topological surface $S$ can one describe the set of all hyperbolic metrics $m$ such that $(S,m)$ maximizes the length of the shortest non-contractible curve on $S$.
If $\textrm{sys}_{g,n}$ denotes the function on $\mathcal{T}_{g,n}$ that gives the length of this shortest curve, one knows that $\textrm{sys}_{g,n}$ has a global maximum by Mumford's compactness theorem \cite{Mumford}, however there are currently very few topological types of surfaces for which we can explicitly describe a global maximum of this function.
In the non-compact setting, Schmutz \cite{Schmutz} showed that some principal congruence subgroups of $\PSL(2,\Z)$ give examples of punctured surfaces which maximize systole and, in the compact setting, the Bolza surface is known to be the unique global maximum of $\textrm{sys}_{2,0}$ in $\mathcal{T}_{2,0}$ \cite{Schmutz2}.
Unfortunately, these surfaces are currently the only explicitly known global maximums of $\textrm{sys}_{g,n}$.

The two types of examples above exhibit many similar and striking properties, for instance they are all arithmetic and they all have large groups of symmetries.
In fact, all of the known maximizers fit into a special class of Riemann surfaces, the so-called congruence surfaces, which are natural candidates to produce many global maximizers of $\textrm{sys}_{g,n}$.
By definition, congruence surfaces are arithmetic and have fundamental groups which contain a principal congruence subgroup (see the definition in Section \ref{sec:hypback}).
These surfaces are good candidates to maximize $\textrm{sys}_{g,n}$ as all congruence surfaces have spectral gap bounded below and are therefore ``big and round'', i.e. they cannot contain a short curve which disconnects them.
Further evidence that these surfaces have the potential to be global maximizers is given by Buser--Sarnak \cite{BuserSarnak} and Katz--Schaps--Vishne \cite{KSV}, who show that, asymptotically, congruence surfaces behave near optimally with respect to systole growth in towers of covers.

Having such a metric on a surface of genus $g$ is a very rare property.
Indeed, combining the aforementioned lower bound on spectral gap with a theorem of Zograf \cite{Zograf} shows that there can only be finitely many congruence surfaces in each genus.
Therefore if one wants to understand whether or not congruence surfaces maximize systole in their respective Teichm\"uller spaces, one needs to first understand these finite sets of surfaces.
For non-compact manifolds, it is a result of Sebbar \cite{SEB} that there are precisely $32$ conjugacy classes of torsion free, congruence subgroups of $\PSL(2,\Z)$ with genus $0$ and hence $32$ congruence surfaces of genus $0$.
However in the compact setting, our knowledge of congruence surfaces is still incomplete.
In fact, it is still an open question of Breuillard and Reid \cite{ThurstonConf} whether such surfaces exist in every genus.

\begin{question}[Breuillard--Reid]\label{Reid}
Let $T$ denote the set of all closed congruence surfaces, then does the set
\[  \{\gen(M) \mid M \in T\} \]
consist of all natural numbers at least $2$?
\end{question}

\noindent Moreover if the answer to Question \ref{Reid} is yes, then the expectation is that these surfaces do not come from the same commensurability class nor trace fields of bounded degree.

%For any surface $M\in T$, let $K_M$ denote its invariant trace field.

\begin{question}[Breuillard--Reid]\label{Reid2}
If Question \ref{Reid} has an affirmative answer, is it true that such surfaces do not come from invariant trace fields $K_M$ of bounded degree?

Perhaps more precisely, suppose Question \ref{Reid} has an affirmative answer.
Then for each natural number $n$ and any choice of $M_n\in T$ with $\gen(M_n)=n$, is it necessarily true that
$$\sup_{n\ge 2}[K_{M_n}:\Q]=\infty?$$
\end{question}

As the surfaces in Question \ref{Reid} and \ref{Reid2} are all compact and arithmetic, all commensurability classes of lattices arise from a particular quaternion algebra construction (see Sections \ref{sec:quatbackground} and \ref{sec:hypback}).
In general these commensurability classes contain a plethora of lattices, including infinitely many maximal lattices, however these commensurability classes always have certain distinguished lattices whose volume, torsion elements, and signature as Fuchsian groups are particularly easy to describe.
These distinguished lattices are constructed using maximal orders in quaternion algebras.

In this paper we investigate Question \ref{Reid} for the subclass of congruence surfaces which have fundamental group given by such a maximal order.
We show that for this subclass of congruence surfaces, Question \ref{Reid} surprisingly has a negative answer.

\begin{theorem}\label{thm:main}
Contingent on the validity of Dokchitser's algorithm to numerically compute L-functions, there is no closed congruence surface constructed from a maximal order of genus 212.
In particular, Question \ref{Reid} restricted to this class of surfaces has a negative answer.
\end{theorem}

\noindent Specifically, the simplest class of congruence surfaces for which one could hope to answer this question does not realize all possible genera.
We remark that this result relies on Dokchitser's algorithm \cite{Dokchitser} to compute L-functions, in particular in the course of proving Theorem \ref{thm:main} we must rely on computer computations of the special zeta value $\zeta_K(-1)$.
Dokchitser's algorithm is the algorithm implemented in both SAGE \cite{Sage} and Magma \cite{Magma} and we use its implementation in the former to compute these values.\\

\noindent \textbf{Acknowledgments:}
This work was carried out during Indiana University's summer REU program and the authors would like to thank Indiana University for their hospitality throughout this time as well as Chris Connell for organizing a fantastic program. 
The authors would also like to thank Alan Reid for enlightening correspondence about the history of this problem.
The authors additionally acknowledge NSF Grant 1461061, without whose funding this work would not have been possible.

%~~~~~~~~~~~~~~~~~~~~~~~~~~~~~~~~~~~~~~~~~~~~~~~~~~~
%~~~~~~~~~~~~~~~~~~~~~~~~~~~~~~~~~~~~~~~~~~~~~~~~~~~
%~~~~~~~~~~~~~~~~~~~~~~~~~~~~~~~~~~~~~~~~~~~~~~~~~~~

\section{Background}
\subsection{Number theoretic preliminaries}\label{sec:ntback}
In this section, we recall some basic facts about number fields and set up notation for the rest of the paper.
The reader is referred to \cite{MR, Marcus, NK} for a thorough treatment of these topics.
Throughout the rest of this paper we will always use $K$ to denote a number field and $\mathcal{O}_K$ its ring of integers. 
Number fields come equipped with $[K:\Q]$ embeddings of $K/\Q$ into $\C$ and we call $K$ \emph{totally real} if each of these embeddings $\sigma$ has the property that $\sigma(K)\subset \R$.
As $\mathcal{O}_K$ is a free abelian group of rank $[K:\Q]$, we may choose a basis $\{\beta_1, \dots, \beta_n\}$ of $\mathcal{O}_K$ over $\Z$ and define the \emph{discriminant} of $K$, henceforth denoted $\Delta_K$, as the quantity
\begin{equation}\label{eqn:disceqn}
\textrm{disc}(\beta_1, \dots, \beta_n) = \det\left((\sigma_i(\beta_j))_{i,j}^2\right),
\end{equation}
where $1\le i\le [K:\Q]$ and each $\sigma_i$ is a distinct embedding of $K/\Q$ into $\C$.
As Equation \eqref{eqn:disceqn} is independent of choice of integral basis, $\Delta_K$ is an invariant of the number field $K$.

Given a number field $K$, a prime ideal $\Fr{p}\subset\mathcal{O}_K$, and a finite extension $L/K$, then $\Fr{p}$ factors in $\mathcal{O}_L$ as
\[\p\intring{L} = \q_1^{e_1} \dots \q_n^{e_n}.\]
We say that each $\q_i$ lies over $\p$ and write $\q_i \mid \p$.
Given a prime $\q_i\mid \p$ as above, we call $e_i$ the \emph{ramification index} and write $e_i = e(\q_i|\p)$.
The ideal $\p$ is said to be \emph{ramified} in $L/K$ if there exists a $q_i$ lying over $\p$ with $e(\q_i|\p) > 1$.
When we talk about a ramified prime of a number field $K$, without reference to an extension, we mean that the prime ramifies in $K/\Q$.
It is a standard result that a rational prime $p$ ramifies in a number field $K$ if and only if $ p \mid \Delta_K$.

Given a finite extension of number fields $L/K$ and prime ideals $\p$, $\q$ of $\mathcal{O}_K$, $\mathcal{O}_L$ (respectively) such that $\q\mid \p$, we have a natural embedding of finite fields $\intring{K}/\p \hookrightarrow \intring{L}/\q$.
We define the \emph{inertial degree} as the degree of this extension and write $f(\q\mid\p)=[\intring{L}/\q:\intring{K}/\p] $.
Given any prime $\p\in\mathcal{O}_K$, the ramification indices and inertial degrees relate to the degree of the extension $L/K$ via the following equation
\begin{equation}\label{eqn:degfromlocal}
\displaystyle[L:K] = \sum_{\q\mid \p} e(\q\mid \p)f(\q\mid \p).
\end{equation}
Given an ideal $I \subset \intring{K}$ we define its norm as the quantity $\Nr(I) = |\intring{K}/I|$.
Notice in particular that for a prime ideal $\p$, $\Nr(\p)=p^f$ where $p$ is the unique rational prime such that $\p\mid p$ and $f=f(\p\mid p)$ its inertial degree.

In Section \ref{section:higherdegcomp} we will require a more detailed statement about the relationship between ramified primes in $K$ and prime divisors of $\Delta_K$.
To this end, we briefly recall that to every number field $K$ we can associate an ideal $\mathcal{D}_K$ of $\mathcal{O}_K$ called the \emph{different ideal} whose definition we omit in favor of simply listing the properties of $\mathcal{D}_K$ that we will need.

\begin{theorem} \label{ram}
For any number field $K$,
\[\Nr(\mathcal{D}_K) = \Delta_K,\]
and a prime $\p \subset \mathcal{O}_K$ is ramified in $K$ if and only if $\p \mid \mathcal{D}_K$.
Moreover if $a$ is the exact power of $\p$ which divides $\mathcal{D}_K$ and if $e = e(\p \mid p)$ is the ramification index of $\p$ over the unique rational prime $p$ with $\p\mid p$, then
\begin{itemize}
\item $a = e - 1$ if $e \not\equiv 0 \pmod p, \textit{ i.e. } \p \text{ is tamely ramified}$,
\item $a \geq e$ if $e \equiv 0 \pmod p, \textit{ i.e. } \p \text{ is wildly ramified}$.
\end{itemize}
As the field norm is multiplicative, this implies that $\Nr(\p)^a\mid \Delta_K$ for the number $a$ above.
\end{theorem}

Using the norm $\Nr(-)$, we can define the \emph{Dedekind zeta function} of $K$ given by the formula
\[\zeta_K(s) = \sum_{I \subset \intring{K}} \frac{1}{\Nr(I)^s},\]
where $s$ is a complex number and the sum runs through all non-zero ideals $I$ in $\intring{K}$.
The Dedekind zeta function also has an Euler product expansion as 
\begin{equation}\label{eqn:Eulerprod}
 \zeta_K(s) = \prod_{\p} \frac{1}{1 - \Nr(\p)^{-s}} ,
\end{equation}
where this product runs through all prime ideals in $\intring{K}$.
Notice that in the specific case of $K=\Q$, $\zeta_{\Q}(s)$ is simply the Riemann zeta function.

Finally, recall that a place $v$ of a number field $K$ is an equivalence class of valuations on $K$.
By a theorem of Onstrowski, all such places are given by $v_{\sigma}(x) = |\sigma(x)|$ for any of the $[K:\Q]$ distinct embeddings $K \hookrightarrow \C$ or by the $\p$-adic valuations where $\p \subset \intring{K}$ is a prime ideal.
We call the former \emph{infinite places} and the latter \emph{finite places}.
For any infinite place $v$, the completion of $K$ with respect to $v$ is isomorphic to either $\R$ or $\C$ and for any finite place $v$, $K_v$ is a finite extension of $\Q_p$.

%~~~~~~~~~~~~~~~~~~~~~~~~~~~~~~~~~~~~~~~~~~~~~~~~~~~
%~~~~~~~~~~~~~~~~~~~~~~~~~~~~~~~~~~~~~~~~~~~~~~~~~~~
%~~~~~~~~~~~~~~~~~~~~~~~~~~~~~~~~~~~~~~~~~~~~~~~~~~~

\subsection{Quaternion algebras over number fields}\label{sec:quatbackground}
By a \emph{quaternion algebra} over a number field $K$ we mean a $4$-dimensional algebra over $K$ with basis $\{1, I, J, IJ\}$ 
and multiplication defined by
\[I^2 = \alpha, \quad J^2 = \beta, \quad IJ = -JI,\]
for some $\alpha,\beta \in K^*$.
Quaternion algebras are central, simple algebras and any 4-dimensional central, simple algebra over $K$ is isomorphic to one of the above form.
We frequently use the compact notation $\quat{\alpha}{\beta}{K}$ to denote the $K$-quaternion algebra where $I^2 = \alpha, J^2 = \beta$.
Given a quaternion algebra $\mathcal{A}=\quat{\alpha}{\beta}{K}$ and an element $a=w+xI+yJ+zIJ\in\mathcal{A}$, we define its norm by the equation
$$\nrm(a)=w^2-\alpha x^2-\beta y^2+\alpha\beta z^2,$$
and write $\mathcal{A}^1$ for the group of norm $1$ elements.
Similarly we define the trace of $a$ as 
$$\tr(a)=2w.$$

Given a quaternion algebra $\A=\quat{\alpha}{\beta}{K}$ and some finite or infinite place $v$, we define the $K_v$-algebra $\A_v$ by the equation
\[\A_v=\A \otimes_{K} K_v .\]
We say that $\A$ \emph{splits} over $v$ if $\A_v\cong \Mat_2(K_v)$ and that $\A_v$ \emph{ramifies} otherwise.
We denote by $\Ram(\A)$ the set of all places for which $\A$ is ramified and note that this set is always finite and of even cardinality.
We use $\Ram_f(\A)$ to denote the places in $\Ram(\A)$ which are finite.
It is a theorem of Albert--Brauer--Hasse--Noether that $\Ram(\A)$ completely determines the algebra $\A$ up to $K$-isomorphism.

For a quaternion algebra $\mathcal{A}$ over $K$ and $L/K$ any extension, we define
$$\mathcal{A}_L=\A\otimes_K L.$$
In the sequel, we will need to understand when a quadratic extension $L/K$ embeds in $\mathcal{A}$, we therefore have the following theorem (see \cite[Theorem 7.3.3]{MR}).
\begin{thm}
Given a quadratic extension $L/K$, the following are equivalent:
\begin{enumerate}
\item $L$ embeds in $\mathcal{A}$,
\item $\mathcal{A}_L\cong\Mat_2(L)$,
\item $L\otimes_K K_v$ is a field for every $v\in\Ram(\A)$.
\end{enumerate}
\end{thm}
\noindent In particular, the last condition is equivalent to requiring that no $v\in\Ram(\A)$ splits in $L/K$.

As with rings of integers of numbers fields, there is a corresponding notion of integral elements in quaternion algebras over number fields.
We call $a \in \A$ an integer if $\intring{K}[a]$ is an $\mathcal{O}_K$-lattice in $\A$, which is equivalent to requiring both $\nrm(a), \tr(a) \in \intring{K}$.
We can then define an \emph{ideal} $I \subset \A$ as a complete $\intring{K}$-lattice and an \emph{order} $\cO$ of $\A$ as an ideal that is a ring with $1$.
As per \cite[\S 6.1]{MR}, orders can be described as rings of integers $\mathcal{O}$ such that $K\mathcal{O}=\A$, from which one can see that there exist \emph{maximal orders} $\mathcal{O}$, i.e. maximal elements in the class of orders with respect to inclusion.

We also need a notion of congruence subgroup for such orders.
Given a quaternion algebra $\A$, a maximal order $\mathcal{O}\subset\A$, and an integral, two-sided ideal $I\subset\mathcal{O}$, we define \emph{the principal congruence subgroup of level $I$} in $\cO^1$ by
\[\cO^1(I) = \{ x \in \cO^1 | x - 1 \in I\}.\]
That is to say $\cO^1(I)$ is the kernel of the homomorphism $\cO^1 \rightarrow (\cO/I)^*$.

%~~~~~~~~~~~~~~~~~~~~~~~~~~~~~~~~~~~~~~~~~~~~~~~~~~~
%~~~~~~~~~~~~~~~~~~~~~~~~~~~~~~~~~~~~~~~~~~~~~~~~~~~
%~~~~~~~~~~~~~~~~~~~~~~~~~~~~~~~~~~~~~~~~~~~~~~~~~~~

\subsection{Hyperbolic surfaces via quaternion algebras}\label{sec:hypback}
We now show how to construct hyperbolic surfaces using quaternion algebras over totally real number fields.
Let $\A$ be a quaternion algebra over a totally real number field $K$, of degree $d$, that is ramified at all but one of its infinite places.
We can and do assume the embedding at which $\A$ splits is the identity.
Then there is an isomorphism of algebras over $\R$ given by
\begin{equation}\label{eqn:quatiso}
\A \otimes_{\Q} \R \cong \Mat_2(\R)\oplus\mathcal{H}^{d-1},
\end{equation}
where $\mathcal{H}$ denotes Hamilton's quaternions \cite[Theorem 8.1.1]{MR}.
Let $\rho: \A \rightarrow \Mat_2(\R)$ be given by the composition of the inclusion map $\A \rightarrow \A \otimes_{\Q} \R$ and the projection onto the first term in Equation \eqref{eqn:quatiso}.
Then given an order $\mathcal{O}$ we can view $\mathcal{O}$ as a subgroup of $\Mat_2(\R)$ via $\rho(\mathcal{O})$.
Under this embedding, we have that for any $a\in\mathcal{A}$, $\nrm(a)=\det(\rho(a))$ and therefore in fact $\rho(\mathcal{O}^1)<\SL_2(\R)$.
Let
$$P: \SL_2(\R) \rightarrow \PSL_2(\R),$$
be the natural quotient map, then $P(\rho(\cO^1)) < \PSL_2(\R)$ for any order $\mathcal{O}\subset\A$.

The subgroup $P(\rho(\cO^1))$ is discrete and of finite covolume and moreover when $\A$ is a division algebra (not isomorphic to matrices), it is cocompact \cite[Theorem 8.1.2]{MR}.
In order to have such an embedding, it is of course necessary that our algebra must split over at least one infinite place, but if our algebra were to split at more than one place, the image of $\cO^1$ in any copy of $\Mat_2(\R)$ would be dense and thus not discrete \cite[Theorem 8.1.2]{MR}.
Consequently the assumption that $\A$ must split at precisely one infinite place is not only sufficient but also necessary.
We call a subgroup $\Gamma<\PSL_2(\R)$ \emph{arithmetic} if it is commensurable with some $P(\rho(\mathcal{O}^1))$ for some $\A$ and $K$ as above.
For the remainder of the paper, when we say a congruence orbifold is \textit{constructed from a maximal order} we mean it is an orbifold given by the quotient $\bH^2/P(\rho(\cO^1))$, where $\rho$ is as above and $\cO^1$ is the group of norm $1$ elements of some maximal order $\cO$ in a division algebra $\A$ over $K$.
When we specify that this orbifold is a congruence surface constructed from a maximal order, we also impose the condition that $P(\rho(\cO^1))$ is torsion free.
Moreover when we want to specify the number field $K$ that the algebra $\A$ giving rise to this congruence surface is defined over, we will call $K$ the \emph{trace field} of $P(\rho(\cO^1))$.
To simplify and slightly abuse notation, in the future we frequently drop our reference to the maps $\rho$ and $P$, so when we write $\surf{\cO^1}$, we implicitly mean $\surf{(P(\rho(\cO^1))}$.
We also will frequently drop reference to the map $\rho$ and write $P(\mathcal{A}^1)$ to mean $P(\rho(\mathcal{A}^1))$.

An arithmetic Fuchsian group $\Gamma$ is a congruence subgroup if it contains a principal congruence subgroup of some level (see the definition in Section \ref{sec:quatbackground}).
We remark that for a maximal order $\cO \subset \A$, $P(\cO^1)$ is necessarily a congruence subgroup as by definition it contains all principal congruence subgroups of $\cO^1$.
However, in general, not all arithmetic Fuchsian groups are congruence subgroups.

The question of when a cocompact arithmetic lattice in $\PSL_2(\R)$ has torsion elements in general is a subtle question and thus the difference between a lattice giving rise to a congruence orbifold versus a congruence surface is a delicate matter.
In particular, the torsion elements of such lattices are frequently not controlled by the arithmetic of the invariant quaternion algebra.
Fortunately in the class of orbifolds which are constructed from maximal orders, there is a much more explicit and tractable relationship \cite[Theorem 12.5.4]{MR}.

\begin{theorem}\label{tor}
Let $\A$ be a quaternion division algebra over a number field $K$ and let $\xi_{2n}$ denote the $2n^{\text{th}}$ roots of unity.
Then the following are equivalent:
\begin{enumerate}
\item The group $P(\A^1)$ contains an element of order $n$,
\item $\xi_{2n} + \xi_{2n}^{-1} \in K, \xi_{2n} \notin K$, and $L = K(\xi_{2n})$ embeds in $\A$,
\item $\xi_{2n} + \xi_{2n}^{-1} \in K, \xi_{2n} \notin K$, and if $\p \in Ram_f(\A)$ then $\p$ does not split in $K(\xi_{2n})$.
\end{enumerate}
Moreover, if $P(\A^1)$ contains an element of $n$-torsion, then for all maximal orders $\cO \subset \A$, $P(\cO^1)$ contains an element of $n$-torsion.
\end{theorem}

\noindent Given a quaternion algebra $\A$ and a number field $K$, we will say that a prime $\p$ \emph{eliminates n-torsion} in $P(\A^1)$ if $\p\in\Ram_f(\A)$ and $\p$ splits in $K(\xi_{2n})$.
Otherwise we say that $\p$ does not help eliminate n-torsion in $\A$.
Notice that given any $n$, such a prime $\p$ is not necessarily unique.

To conclude this section, we provide Borel's formula for the volume of $\Hy^2/\mathcal{O}^1$ where $\cO \subset \A$ is a maximal order (see \cite{BOR} or \cite[Theorem 11.1.1]{MR}).
This formula will be an indispensable tool in our ability to study the genera of congruence surfaces from maximal orders.
\begin{theorem}[Borel] \label{thm:1}
Let $K$ be a totally real number field, $\A$ a $K$-quaternion algebra that is ramified at all but one of its infinite places, $\cO \subset \A$ some maximal order. Then the volume of $\surf{\cO^1}$ is given by
\begin{equation} \label{eq:1}
\vol(\Hy^2/\mathcal{O}^1)=\frac{8\pi\Delta_K^{3/2}\zeta_K(2)}{(4\pi^2)^{[K:\Q]}}\prod_{\substack{\p \in \Ram_f(\A)}}(\Nr(\p) - 1).
\end{equation}
\end{theorem}

%~~~~~~~~~~~~~~~~~~~~~~~~~~~~~~~~~~~~~~~~~~~~~~~~~~~
%~~~~~~~~~~~~~~~~~~~~~~~~~~~~~~~~~~~~~~~~~~~~~~~~~~~
%~~~~~~~~~~~~~~~~~~~~~~~~~~~~~~~~~~~~~~~~~~~~~~~~~~~

\section{Dedekind Zeta Function Bounds and Computation}\label{sec:zetabounds}

\noindent In order to understand the genera that can be achieved by a congruence surface constructed from a maximal order, we must first understand the values of the Dedekind zeta function for arbitrary number fields in Equation \eqref{eq:1}.
To do this, we rewrite Equation \eqref{eq:1} in terms of the special value $\zeta_K(-1)$ using the functional equation for the Dedekind zeta function.
We then go on to use bounds of Odlyzko to reduce the possible trace fields $K$ of such a surface to a finite, computable list.

When $K$ is a totally real number field of degree $d$, the Dedekind zeta function satisfies the functional equation \cite[Pg. 467]{NK} (see also \cite[Pg. 346]{MR})
$$\frac{|\Delta_K|^{s/2}\Gamma(s/2)^d}{\pi^{sd/2}}\zeta_K(s)=\frac{|\Delta_K|^{(1-s)/2}\Gamma((1-s)/2)^d}{\pi^{(1-s)d/2}}\zeta_K(1-s),$$
where $\Gamma(s)$ is the Gamma function.
Evaluating at $s=2$ gives that
\begin{equation}\label{eqn:zetatransfer}
\zeta_K(2)=\frac{(-2)^d\pi^{2d}}{\Delta_K^{3/2}}\zeta_K(-1),
\end{equation}
and since $(-2)^d$ has the same sign parity as $\zeta_K(-1)$, Equation \eqref{eq:1} can be written as
\[\vol(\surf{\cO^1}) = \frac{\pi|\zeta_K(-1)|}{2^{d - 3}} \prod_{\mathfrak{p} \in \Ram_f(\A)}(\Nr(\mathfrak{p}) - 1),\]
which is the formulation we will use for the rest of the paper.
When $\Hy^2/\mathcal{O}^1$ is a hyperbolic surface, the Gauss--Bonnet theorem shows that its area is
$$\vol(\Hy^2/\mathcal{O}^1)=4\pi(g-1),$$
and therefore we conclude that
\begin{equation} \label{eq:3}
g - 1 = \frac{|\zeta_K(-1)|}{2^{d - 1}}\prod_{\mathfrak{p} \in \Ram_f(\A)}(\Nr(\mathfrak{p}) - 1).
\end{equation}

In the rest of this section, we will give bounds on the quantity preceding the product in Equation \eqref{eq:3} to show that a genus $212$ congruence surface must have trace field of degree at most $10$.
The starting point of this is the following well known proposition, which comes directly from an analysis of $\zeta_K(s)$ in its Euler product form (see Equation \eqref{eqn:Eulerprod}).

\begin{prop} \label{prop:3}
For any number field $K$ of degree $d$ we have 
\[ 1 \leq \zeta_K(2) \leq \zeta(2)^{d},\]
where $\zeta(2)$ denotes the Riemann zeta function evaluated at $2$.
\end{prop}

\noindent Using Equation \eqref{eqn:zetatransfer} and the fact that $\zeta(2)=\frac{\pi^2}{6}$ allows us to convert this inequality to one for $\frac{|\zeta_K(-1)|}{2^{d - 1}}$.
In particular, we obtain
\begin{equation} \label{eq:4}
\frac{2\Delta_K^{3/2}}{\pi^{2d]}4^{d}} \leq \frac{|\zeta_K(-1)|}{2^{d - 1}} \leq \frac{\Delta_K^{3/2}}{2^{d - 1}12^{d}}.
\end{equation}
We also have the following bounds on the discriminant of a totally real number field $K$ from Takeuchi \cite[Proposition 2.3]{TK}, originally due to Odlyzko \cite{OD}.

\begin{proposition}\label{TAK}
Let $K$ be a totally real number field, then
\[\Delta_K > a^{d}e^{-b},\]
where a = 29.099, b=8.3185
\end{proposition}

Combining Proposition \ref{TAK} with Equation \eqref{eq:4} we therefore have the following bound on the degree of an trace field.

\begin{prop}\label{prop:degupperbnd}
Let $M=\Hy^2/\mathcal{O}^1$ be a congruence surface constructed from a maximal order of genus $212$, then the trace field $K$ of $M$ has degree at most $10$.
\end{prop}
\begin{proof}
Let $d=[K:\Q]$ and define $C(d)=ae^{-b/n}$ where $a$, $b$ are as in Proposition \ref{TAK}, so that
$$C(d)\le\Delta_K^{1/d},$$
for any totally real number field $K$ of degree $d$.
Using Equation \eqref{eq:3}, we also see that 
$$\frac{2\Delta_K^{3/2}}{\pi^{2d}4^{d}} \leq \frac{|\zeta_K(-1)|}{2^{d - 1}}\prod_{\mathfrak{p} \in \Ram_f(\A)}(\Nr(\mathfrak{p}) - 1)=211,$$
and therefore
$$\Delta_K^{1/d}\le\left(4\pi^2\left(\frac{211}{2}\right)^{1/d}\right)^{2/3}=D(d).$$
It follows that the inequality $C(d)\le D(d)$ holds if and only if $d\le 12$, where we have collected the relevant truncated values of $C(d)$, $D(d)$ in Table \ref{table:Odlyz}.
When $d\in\{11,12\}$ we note that work of Voight and Jones--Roberts \cite{Voight,JR} (see also \cite{LMFDB}) shows that in fact there is no totally real number field with
$$\Delta_K^{1/d}\le D(d)$$
and from this we conclude that the inequality holds if and only if $d\le 10$.
\end{proof}

\begin{table}[t]
\centering
\begin{tabular}{||c |  c | c ||} \hline
$[K:\Q]$ &  $C(d)$ & $D(d)$  \\
\hline
5                                        & 5.12        &21.57         \\
\hline
6                                     & 7.27                  & 19.45                                \\
\hline
7                                    & 8.86                      & 18.06                            \\
\hline
8                                    & 10.28                 & 17.09                                 \\
\hline
9                                  & 11.54                & 16.37                                  \\
\hline
10                                & 12.66                & 15.81                                  \\
\hline
11                                & 13.66    & 15.37   \\
\hline
12                                 & 14.54   & 15.01   \\
\hline                    
13                                 & 15.34   & 14.72   \\
\hline
\end{tabular}
\caption{Lower and upper bounds for the root discriminant in fixed degree}
\label{table:Odlyz}
\end{table}

In fact the proof of Proposition \ref{prop:degupperbnd} shows something stronger, namely that we are able to reduce the possible list of trace fields to a finite, computable list.

\begin{cor}\label{cor:zetalist}
There are only finitely many possible trace fields for a congruence surface constructed from a maximal order of genus $212$.
Moreover this list is finite, computable, and for each field $K$ in this list $\zeta_K(-1)$ is computable via Dokchitser's algorithm.
\end{cor}

\begin{proof}
Indeed using the bound on the root discriminant $D(d)$ above, we find that in each degree there are only finitely many fields $K$ such that $\Delta_K^{1/d} \leq D(d)$.
From this, one can use the work of work of Voight and Jones--Roberts \cite{Voight,JR} to explicitly list all such number fields.
SAGE's ``zeta\_function()'' command \cite{Sage} then gives the relevant values of $\zeta_K(-1)$ for these fields.
\end{proof}

\noindent We have collected the list of all such fields with their relevant values of $\zeta_K(-1)$ \textcolor{blue}{\href{http://pages.iu.edu/~nimimill/zetatables.pdf}{here}}.
The table is extremely long and cannot be included in this paper, however we will pick out the relevant fields for the paper in the following sections and in the appendix.

%% \begin{table}[t]
%% \centering
%% \begin{tabular}{||c | c | c | c ||} \hline
%% $[K:\Q]$ & $\Delta_K$ & $C(d)$ & $D(d)$  \\
%% \hline
%% 5                      & 5089                  & 5.12        &21.57         \\
%% \hline
%% 6                      & 148111               & 7.27                  & 19.45                                \\
%% \hline
%% 7                      & 4309906              & 8.86                      & 18.06                            \\
%% \hline
%% 8                      & 125413956              & 10.28                 & 17.09                                 \\
%% \hline
%% 9                      & 3649420715            & 11.54                & 16.37                                  \\
%% \hline
%% 10                     & 106194493407           & 12.66                & 15.81                                  \\
%% \hline
%% 11                     & 3090153563670           & 13.66    & 15.37   \\
%% \hline
%% 12                     & 89920378549235            & 14.54   & 15.01   \\
%% \hline                    
%% 13                     & 2616593095404191            & 15.34   & 14.72   \\
%% \hline
%% \end{tabular}
%% \caption{Lower and upper bounds for the discriminant and root discriminant in fixed degree}
%% \label{table:Odlyz}
%% \end{table}

%~~~~~~~~~~~~~~~~~~~~~~~~~~~~~~~~~~~~~~~~~~~~~~~~~~~
%~~~~~~~~~~~~~~~~~~~~~~~~~~~~~~~~~~~~~~~~~~~~~~~~~~~
%~~~~~~~~~~~~~~~~~~~~~~~~~~~~~~~~~~~~~~~~~~~~~~~~~~~

\section{Congruence Conditions on Torsion in Quaternion Algebras}

\noindent We now explore some congruence conditions on primes in totally real number fields $K$.
These congruence conditions will be useful in connecting splitting properties of primes $\Fr{p}$ in $\Ram_f(\A)$ to torsion in $P(\mathcal{A}^1)$ over $K$ via Theorem \ref{tor}.
In particular, they will allow us to show that many trace fields cannot support congruence surfaces from maximal orders of genus $212$, including $\Q$.

For this, we first give some preliminary lemmas controlling the inertial degree of primes in the fields we are interested in.

\begin{lemma} \label{lem:2split}
Let $K$ be a totally real number field and $\p \subset \mathcal{O}_K$ a non-dyadic prime ideal.
If $\p$ splits in $K(i)/K$ then $\Nr(\p) \equiv 1 \pmod 4$.
\end{lemma} 
\begin{proof}
Let $p$ be the unique rational prime such that $\p\mid p$.
The lemma is then clear if $p \equiv 1 \pmod 4$, as $\Nr(\p)=p^f$ for some $f\in\N$.
By the non-dyadic assumption, it therefore suffices to consider when $p \equiv 3 \pmod 4$, where we point out that $\Nr(\p) \equiv 1 \pmod 4$ if and only if $f(\p|p) \equiv 0 \pmod 2$.
As $\p$ splits in $K(i)$, there are primes $\q,\overline{\q}$ of $K(i)$ lying over $\p$ such that
$$1 = f(\q\mid \p)=f(\overline{\q}\mid \p).$$
Moreover inertial degrees are multiplicative in towers and since $
\Q\subset\Q(i)\subset K(i)$ and $p\equiv 3\pmod{4}$, there is an inert prime $\p'$ of $\Q(i)/\Q$ such that
$$f(\q\mid p)=f(\q\mid\p')f(\p'\mid p)=2f(\q\mid \p'),$$
with a similar statement for $\overline{\q}$.
In particular $f(\q\mid p)=f(\p\mid p)$ is even, giving the lemma.
\end{proof}

\begin{lemma}\label{lem:3split}
Let $K$ be a totally real number field and $\p \subset \mathcal{O}_K$ a prime ideal not lying over the rational prime $3$.
If $\p$ splits in $K(\sqrt{-3})/K$ then $\Nr(\p) \equiv 1 \pmod 3$.
\end{lemma} 
\begin{proof}
The proof is identical to that of Lemma \ref{lem:2split}, except one replaces the condition that $p\equiv 1\pmod{4}$ by the condition that $p\equiv 1\pmod{3}$.
\end{proof}

\noindent We are also interested in the ramification index of primes in these fields.
We therefore have the following two lemmas, concerning dyadic and triadic primes.

\begin{lemma} \label{lem:2ram}
Suppose $\p$ is a dyadic prime in a number field $K$ that splits in $K(i)/K$, then $e(\p | 2) \equiv 0 \pmod 2$.
\end{lemma}
\begin{proof}
As 2 ramifies in $\Q(i)$, we have that $e(\q|2) \equiv 0 \pmod 2$ for any prime $\q$ in $K(i)$ lying over $\p$.
By assumption $\p = \q\overline{\q}$ in $K(i)/K$ giving $e(\q|\p) = e(\overline{q}|\p) = 1$.
By multiplicativity of the ramification index in towers it then must hold that $e(\p|2) \equiv 0 \pmod 2$.
\end{proof}

\begin{lemma}\label{lem:3ram}
Suppose $\p$ is a prime lying over the rational prime $3$ in a number field $K$ that splits in $K(\sqrt{-3})/K$, then $e(\p | 2) \equiv 0 \pmod 2$.
\end{lemma}
\begin{proof}
The proof is identical to Lemma \ref{lem:2ram} with $K(i)$ replaced by $K(\sqrt{-3})$.
\end{proof}

At this point, we are in a position to eliminate the case of trace field $\Q$.

\begin{prop}\label{prop:noQ}
There is no congruence surface constructed from a maximal order of genus $212$ with trace field $\Q$.
\end{prop}
\begin{proof}
When $K=\Q$, Equation \eqref{eq:3} shows that
$$211 = \frac{1}{12}\prod_{p \in \Ram_f(\A)}(p - 1),$$
for a quaternion algebra $\mathcal{A}$ over $\Q$.
The condition that $\mathcal{A}$ supports a congruence surface constructed from a maximal order is the condition that $\Q(\xi_{2n})$ does not embed in $\A$ for any natural number $n$ (see Theorem \ref{tor}).
Equivalently, this happens if and only if for every $n$ such that $\cos(\pi/n)\in \Q$, any $p\in\Ram_f(\A)$ splits in $\Q(\xi_{2n})$.
Critically, $\cos(\pi/2),\cos(\pi/3)\in\Q$ and we must therefore ensure that there is at least one prime in $\Ram_f(\A)$ that splits in $\Q(\xi_4)$ (resp. $\Q(\xi_6)$).

Since no primes ramify in $\Q$, Lemmas \ref{lem:2split}--\!\! \ref{lem:3ram} show that this is equivalent to either having a single prime $p\in\Ram_f(\A)$ such that $p\equiv 1\pmod{12}$ or distinct primes $p,q\in\Ram_f(\A)$ such that $p\equiv 1\pmod{3}$ and $q\equiv 1\pmod{4}$.
We will show that neither of these are possible.
Indeed in the former case, writing $p=12 k+1$ for some $k>0$, we see that
$$ 211 = k \prod_{q \in \Ram_f(\A)\backslash \{p\}}(q - 1),$$
and therefore $k=1$ or $k=211$.
However, both choices of $k$ imply that either $211+1=212$ or $12(211)+1=2533$ is a prime, which is clearly false.
In the latter case, write $p=3k+1$ and $q=4k'+1$, then again we see that
$$ 211 = kk' \prod_{\ell \in \Ram_f(\A)\backslash \{p,q\}}(\ell - 1),$$
where we include the possibility that $\Ram_f(\A)=\{p,q\}$ and the latter product is over an empty set.
By the above $k\in\{1,211\}$, but if $k=1$ then $p=3k+1=4$ is not prime and if $k=211$ then $p=3(211)+1$ is even and so also not prime.
Therefore we conclude that there is no such quaternion algebra $\A$ over $\Q$ and hence no such congruence surface, completing the proof.
\end{proof}

We conclude this section by noting that one can vastly generalize the analysis of Proposition \ref{prop:noQ} to be able to handle fields $K$ such that $\frac{|\zeta_K(-1)|}{2^{[K:\Q] - 1}} = \frac{1}{n}$ for some $n \in \N$.
This will be extremely useful for us in ruling out congruence surfaces constructed from maximal orders where the ramification of the quaternion algebra contributes $211$ to the righthand side of Equation \eqref{eq:3} via its ramification set.
In particular the conditions in the theorem mimic constraints coming from our computation of the finite list of all possible trace fields, as mentioned at the end of Section \ref{sec:zetabounds}.
The statement of the theorem is a bit involved, however we will need all of the possibilities listed therein.

\begin{theorem}\label{1/n}
Assume that there is a congruence surface constructed from a maximal order of genus $212$ such that $\frac{|\zeta_K(-1)|}{2^{[K:\Q] - 1}} = \frac{1}{n}$, where $n$ is any natural number $2\le n\le 84$.
Then $n$ is divisible by a natural number $k$ such that $k\in\{10,22,42,52,58,70,72\}$.

Moreover if $2\nmid\Delta_K$, then we must have that $n\in\{72,80\}$.
In the case that $n=72$, the associated quaternion algebra $\A$ has finite ramification set
$$\Ram_f(\A)=\{\p,\q_1,\dots,\q_s\},$$
where $\Nr(\p)=15193$ and $\Nr(\q_i)=2$ for all $i$.
When $n=80$, then $9\mid\Delta_K$ and the associated quaternion algebra $\A$ has finite ramification
$$\Ram_f(\A)=\{\p,\p',\q_1,\dots,\q_s\},$$
with $\Nr(\p)=2111$, $\Nr(\p')\in\{3,9\}$, and $\Nr(\q_i)=2$ for all $i$.

In the above equations we allow for the possibility that $s=0$ and therefore $\Ram_f(\A)=\{\p\}$ (respectively $\Ram_f(\A)=\{\p,\p'\}$).
\end{theorem}
\begin{proof}
Equation \eqref{eq:3} shows that
$$211n = \prod_{\p \in \Ram_f(\A)}(\Nr(\p) - 1),$$
and therefore $\mathcal{A}$ must ramify at some prime $\p$ such that $\Nr(\p) = 211k + 1$ for which $k \mid n$ and $\Nr(\p)$ is a prime power.
One can easily check that the only such $k$ are $k\in\{10,22,42,52,58,70,72\}$.
This gives the first claim.

For the second, first assume that $k\equiv 2\pmod{4}$, which is the case for all $k\neq 52,72$ in the above set.
Then $\Nr(\p)\equiv 3\pmod{4}$ and therefore by Lemma \ref{lem:2split} ramifying at such a prime does not eliminate $2$-torsion in $P(\A^1)$.
Moreover when $2 \nmid \Delta_K$, Lemmas \ref{lem:2split} and \ref{lem:2ram} show that $\A$ must ramify at some $\q$ such that $\Nr(\q) \equiv 1 \pmod 4$ in order to eliminate $2$-torsion in $P(\A^1)$.
Writing $\Nr(\q) = 4k'+1$, we then have 
$$211n= (211k)4k' \prod_{\mathfrak{r} \in \Ram_f(\A)\setminus \{\p,\q\}}(\Nr(\mathfrak{r}) - 1),$$
where the latter product is potentially over an empty set.
In particular, since each $k$ is even, we see that $n\equiv 0\pmod 8$.
This is impossible in all cases except for when $k=10$ and $n\in\{40,80\}$, where we then have that $k'\in\{1,2\}$.

If $n=40$, then $k'=1$ by the above equation and consequently $\Nr(\q)=5$.
By Lemma \ref{lem:3split} such a prime cannot eliminate $3$-torsion in $P(\A^1)$.
As all other primes $\Fr{r}$ in $\Ram_f(\A)$ must have $\Nr(\Fr{r})=2$, we see that in fact no prime in $\Ram_f(\A)$ eliminates $3$-torsion in $P(\A^1)$.
When $n=80$, $k'\in\{1,2\}$ and hence either $\Nr(\q)=9$ or $\Nr(\q)=5$ and there is some prime $\p'\in\Ram_f(\A)$ that has $\Nr(\p')=3$ and eliminates $3$-torsion in $P(\A^1)$.
In such a setting, Lemma \ref{lem:3ram} shows that the ramification index of such a prime must be contained in $2\N$.
Consequently Theorem \ref{ram} shows that $9\mid\Delta_K$ as required.

The remaining cases to analyze are when $n=k\in\{52,72\}$.
In this setting we have 
$$1= \prod_{\mathfrak{q} \in \Ram_f(\A)\setminus \{\p\}}(\Nr(\mathfrak{\q}) - 1).$$
Therefore all $\q\in\Ram_f(\A)\setminus\{\p\}$ must have $\Nr(\q)=2$.
Notice that if $k=n=52$, $\Nr(\p)\equiv 2 \pmod{3}$ and therefore Lemma \ref{lem:3split} shows that $\p$ does not eliminate $3$-torsion in $P(\A^1)$.
Lemmas \ref{lem:3split} and \ref{lem:3ram} then show immediately that $P(\A^1)$ will always contain $3$-torsion for such a ramification set.
When $k=n=72$, $\Nr(\p)\equiv 1\pmod{12}$ which eliminates both $2$ and $3$-torsion in $P(\A^1)$, as required.
\end{proof}

From our computed list of number fields, we collect in Table \ref{table:badzeta} the fields such that $ \frac{|\zeta_K(-1)|}{2} = \frac{211}{n}$ for some $n>1$ or $ \frac{|\zeta_K(-1)|}{2} = \frac{1}{n}$ for some $n$ such that $n$ is divisible by some $ k \in \{10,22,42,52,58,70,72\}$, since Equation \eqref{eq:3} and Theorem \ref{1/n} show that only these fields could yield a genus 212 surface.
Furthermore, we remark that Theorem \ref{1/n} shows the impossibilty of genus 212 surface over any of the fields with $ \frac{|\zeta_K(-1)|}{2} =  \frac{1}{n} $ in Table \ref{table:badzeta}. 
The remaining sections will rule out such a construction over all other fields. 

\begin{rem}\label{rem:birchtate}
If one assumes the Birch--Tate conjecture, then the above analysis can be made in a much more systematic way.
In particular, the Birch--Tate conjecture asserts that for totally real number fields $K$
\begin{equation}\label{eqn:birchtate}
|\zeta_K(-1)|=\frac{|K_2(\mathcal{O}_K)|}{w_2(K)},
\end{equation}
where
$$w_2(K)=\max\{n\in\N\mid K(\xi_n)/K\text{ is a 2-elementary abelian extension}\}.$$
This conjecture is a theorem when $K$ is an abelian extension of $\Q$ and for general number fields is known to hold up to a power of $2$.
Results of Tate \cite[Theorems 6.1 and 6.3]{Tate} show that the numerator of Equation \eqref{eqn:birchtate} is divisible by $2^{[K:\Q]}$ and moreover give a general method for computing the power of $2$ which divides the numerator of Equation \eqref{eqn:birchtate}.
In particular, writing
$$\frac{|\zeta_K(-1)|}{2^{[K:\Q]}}=\frac{s}{t},$$
in reduced form and assuming the Birch--Tate conjecture we see that $t\mid w_2(K)$.
Moreover it is then apparent that if $K\cap \Q(\xi_{n})=\Q$ for all $n\in\N$, then $t\mid 12$.
When $K\cap \Q(\xi_{n})\neq\Q$ for some $n\in\N$, it will produce extra congruence conditions on $\Ram_f(\A)$ for $P(\A^1)$ to be torsion free, as there will then be cyclotomic extensions $\Q(\xi_{2n'})$ for $n'\neq 2,3$ such that $[K(\xi_{2n'}):K]= 2$.
One can see this appear explicitly in the analysis of Section \ref{sec:quadfield} when we deal with real quadratic extensions (notice that the Birch--Tate conjecture is a theorem in this setting since the extension is abelian).
Equation \eqref{eqn:birchtate} is one reason for the restriction on the $n$ we consider in Theorem \ref{1/n}, since these are relevant denominators when $[K:\Q]\le 10$.
\end{rem}

%~~~~~~~~~~~~~~~~~~~~~~~~~~~~~~~~~~~~~~~~~~~~~~~~~~~
%~~~~~~~~~~~~~~~~~~~~~~~~~~~~~~~~~~~~~~~~~~~~~~~~~~~
%~~~~~~~~~~~~~~~~~~~~~~~~~~~~~~~~~~~~~~~~~~~~~~~~~~~

\section{The case of quadratic trace field}\label{sec:quadfield}
In this section, we show that there is no genus 212 congruence surface constructed from a maximal order with quadratic trace field.
To do so we use work of Zagier \cite{DZ}, who showed that there is an explicit formula for $\zeta_K(-1)$ for all real quadratics $K=\Q(\sqrt{d})$.
Moreover Zagier shows that $\frac{\zeta_K(-1)}{2}$ is a positive rational number with denominator dividing $24$ except in the case that $K$ is $\Q(\sqrt{5})$.
It turns out that this denominator divides $12$ unless $\Q(\sqrt{2})$ or $\Q(\sqrt{5})$ (see Remark \ref{rem:quadfield}).

To begin, we first rule out all of the other fields by showing that if the denominator of $\frac{\zeta_K(-1)}{2}$ divides $12$ then there cannot be such a congruence surface.
We then go on to rule out the remaining two exceptional quadratic fields.

\begin{proposition}\label{quadprop}
There is no congruence surface constructed from a maximal order of genus $212$ with trace field $K$ a real quadratic, when $\frac{\zeta_K(-1)}{2} = \frac{211}{n}$ for some $n \mid 12$, $n<12$.
\end{proposition}
\begin{proof}
Begin by noting that when $K$ is real quadratic and $\A$ gives rise to a hyperbolic manifold, the discussion in Sections \ref{sec:quatbackground} and \ref{sec:hypback} shows that $\A$ ramifies at one of its infinite places and the cardinality of the finite ramification, $|\Ram_f(\A)|$, is odd.
Therefore Equation \eqref{eq:3} becomes
\begin{equation} \label{eq:5}
\prod_{\mathfrak{p} \in \Ram_f(\A)}(\Nr(\mathfrak{p}) - 1) = n,
\end{equation}
Notice that if $n=1$, then every $\p\in\Ram_f(\A)$ has the property that $\Nr(\p)=2$.
In particular, Lemma \ref{lem:3split} combined with Theorem \ref{tor} shows that, in this case, $P(\A^1)$ has $3$-torsion.
Consequently it suffices to assume $n\mid 6$ and $n>1$.

First suppose that $\frac{\zeta_K(-1)}{2} = \frac{211}{6}$. 
Then Equation \eqref{eq:5} and the odd cardinality of $\Ram_f(\A)$ show that we must be in one of the following scenarios:
\begin{itemize}
\item $\Ram_f(\A)=\{\p\}$, where $\Nr(\p)=7$,
\item $\Ram_f(\A) = \{\p, \q_1, \q_2\}$, where $\Nr(\p) = 7$ and $\Nr(\q_j) = 2$,
\item $\Ram_f(\A) = \{\p_1, \p_2,\q\}$, where $\Nr(\p_1) = 4$, $\Nr(\p_2) = 3$, and $\Nr(\q) = 2$.
\end{itemize}
In the first case, since $7\equiv 3\pmod{4}$, Lemma \ref{lem:2split} shows that $\p$ does not split in $K(i)/K$.
Therefore by Theorem \ref{tor}, $P(\mathcal{A}^1)$ will still contain $2$-torsion and hence does not give rise to a congruence surface.
In the second case, again $\p$ will not help eliminate $2$-torsion and so we must have that $\q_j$ splits in $K(i)/K$ for some $j$.
However by Lemma \ref{lem:2ram} this shows that one of the $\q_j$ must ramify in $K$ which by Equation \eqref{eqn:degfromlocal} gives that $[K:\Q]\ge 3$, a contradiction.
Now assume we are in the third case, then again Equation \eqref{eqn:degfromlocal} shows that having the primes $\p_1$, $\q$ lying over $2$ is impossible for a field of degree $[K:\Q]\le 2$.

Next, suppose $\frac{\zeta_K(-1)}{2} = \frac{211}{4}$. 
Then by Equation \eqref{eq:5} we must see one of the following
\begin{itemize}
\item $\Ram_f(\A) = \{\p\}$, where $\Nr(\p) = 5$,
\item $\Ram_f(\A) = \{\p, \q_1, \q_{2}\}$, where $\Nr(\p) = 5$ and $\Nr(\q_j) = 2$,
\item $\Ram_f(\A) = \{\p_1, \p_2,\q\}$, where $\Nr(\p_1) = \Nr(\p_2) = 3$ and $\Nr(\q) = 2$.
\end{itemize}
To eliminate the first two cases, simply note that by Lemma \ref{lem:3split} no primes in $\Ram_f(\A)$ split in $K(\sqrt{-3})/K$, and thus a ramification set of this form gives $3$-torsion in $P(\A^1)$ by Theorem \ref{tor}.
In the latter case, $\Nr(\p_1)= \Nr(\p_2) = 3$ implies that there are two primes lying over $3$ and by Lemma \ref{lem:3ram} these primes do not split in $K(\sqrt{-3})/K$.
Again this ramification set will not eliminate $3$-torsion in $P(\A^1)$.

Now, suppose $\frac{\zeta_K(-1)}{2} = \frac{211}{3}$.
Then we must see one of
\begin{itemize}
\item $\Ram_f(\A) = \{\p\}$, where $\Nr(\p) = 4$,
\item $\Ram_f(\A) = \{\p, \q_1, \q_{2}\}$, where $\Nr(\p) = 4$ and $\Nr(\q_j) = 2$.
\end{itemize}
The latter is clearly absurd in a real quadratic so we assume $\Ram_f(\A)=\{\p\}$, where $\p$ is an inert prime dividing $2$.
By Lemma \ref{lem:2ram}, $\p$ does not split in $K(i)/K$ so this ramification set gives $2$-torsion in $P(\A^1)$ by Theorem \ref{tor} and hence $\A$ does not support a congruence surface from a maximal order.

Finally suppose  $\frac{\zeta_K(-1)}{2} = \frac{211}{2}$.
Then we must see one of
\begin{itemize}
\item $\Ram_f(\A) = \{\p\}$, where $\Nr(\p) = 3$,
\item $\Ram_f(\A) = \{\p, \q_1,  \q_{2}\}$, where $\Nr(\p) = 3$ and $\Nr(\q_j) = 2$.
\end{itemize}
By Lemma \ref{lem:2split}, $\p$ does not split in $K(i)/K$ and therefore we can immediately rule out the former case.
In the latter case this also shows that one of the $\q_j$ must split in $K(i)/K$. 
However again by Lemma \ref{lem:2ram}, such a $\q_j$ must ramify in $K$ contradicting the existence of both $\q_1$ and $\q_2$ via Equation \eqref{eqn:degfromlocal}.
Therefore no congruence surface can be built with this ramification set.
\end{proof}

\begin{lemma}\label{quadtwelvelemma}
There is no real quadratic number field $K$ with $\frac{\zeta_K(-1)}{2} = \frac{211}{12}$.
\end{lemma}
\begin{proof}
One can verify this completely computationally.
Indeed using the bounds from Equation \eqref{eq:4} we get that $\frac{\zeta_K(-1)}{2} > \frac{211}{12}$ whenever $\Delta_K > 574$.
Moreover, Zagier \cite[Equations (15)--(17)]{DZ} shows that for real quadratics $K=\Q(\sqrt{d})$
\begin{equation}\label{eqn:Zagierzeta}
\displaystyle\frac{\zeta_K(-1)}{2}=\frac{1}{120}e_1(\Delta_K)=\frac{1}{120}\sum_{\substack{x^2\equiv \Delta_K\pmod{4}\\|x|\le \sqrt{n}}}\sigma_1\left(\frac{\Delta_K-x^2}{4}\right),
\end{equation}
where $\sigma_1(-)$ denotes the sum of positive divisors function and $\Delta_K$ is the discriminant of $\Q(\sqrt{d})$, i.e.
\[ \Delta_K =
\begin{cases}
d, & d  \equiv 1 \pmod 4, \\
4d, &  d \equiv 2, 3 \pmod 4.
\end{cases}
\]
Computationally one can easily list all real quadratic number fields with $\Delta_K$ up to $574$, which we do in Table \ref{table:quadfield} in the Appendix.
By inspection, one then sees that no such $K$ exists.
\end{proof}

We may rule out the remaining exceptional fields using the analysis of Theorem \ref{1/n}.

\begin{prop}\label{prop:noexceptionalquads}
Let $K=\Q(\sqrt{d})$ for $d\in\{2,5\}$, then there is no genus $212$ congruence surface constructed from a maximal order.
\end{prop}
\begin{proof}
Indeed when $K=\Q(\sqrt{2})$, $\frac{\zeta_K(-1)}{2}=\frac{1}{24}$ and the result is immediate from Theorem \ref{1/n}.
When $K=\Q(\sqrt{5})$, the discriminant is given by $\Delta_K=5$ and so the result again follows from Theorem \ref{1/n} in the case that $2\nmid\Delta_K$.
\end{proof}

\begin{rem}\label{rem:quadfield}
One can give a more specific examination of Equation \eqref{eqn:Zagierzeta} to find precisely which fields $K=\Q(\sqrt{d})$ have the potential to have a denominator of $12$ in $\zeta_K(-1)$.
For instance one can prove that, outside of the exceptional cases of $d\in\{2,5\}$, $10$ always divides
$$\sum_{\substack{x^2\equiv \Delta_K\pmod{4}\\|x|\le \sqrt{n}}}\sigma_1\left(\frac{\Delta_K-x^2}{4}\right),$$
and moreover $20$ divides this sum if and only if $d$ is not of the form $d=p$ or $d=2p$ where $p$ is a prime congruent to $3,5\pmod{8}$.
Unfortunately even with these constraints, it is difficult to find a general method for ruling out $\frac{211}{12}$ as a value of $\frac{\zeta_K(-1)}{2}$ without simply computing all possible values.
In fact, this is one reason that we chose genus $212$ as a candidate for the Theorem \ref{thm:main}.
\end{rem}

Combining the discussion in this section with Propositions \ref{quadprop}, \ref{prop:noexceptionalquads}, and Lemma \ref{quadtwelvelemma} we then have the following theorem.

\begin{thm}\label{thm:noquad}
There is no congruence surface constructed from a maximal of genus $212$ with real quadratic trace field.
\end{thm}
\begin{proof}
By Equation \eqref{eq:3} we have that
$$211=\frac{\zeta_K(-1)}{2}\prod_{\p\in\Ram_f(\A)}(\Nr(\p)-1).$$
By Lemma \ref{prop:noexceptionalquads} we may assume that $K=\Q(\sqrt{d})$ for $d\neq 2,5$.
For any such field, we claim that the denominator of $\frac{\zeta_K(-1)}{2}$ must divide $12$.
Indeed there are three ways to see this: the first is by carrying out the analysis in \cite[\S 3]{DZ} and analyzing the subsequent divisibility conditions, the second is by applying the Birch--Tate conjecture (a theorem in this case since $K/\Q$ is abelian) where when $d\neq 2,5$ we see that $w_2(K)=24$ and hence the denominator divides $12$, and the third is simply by examining the values in Table \ref{table:quadfield}.
From this, one can apply Proposition \ref{quadprop} and Lemma \ref{quadtwelvelemma} to conclude the theorem.
\end{proof}

%~~~~~~~~~~~~~~~~~~~~~~~~~~~~~~~~~~~~~~~~~~~~~~~~~~~
%~~~~~~~~~~~~~~~~~~~~~~~~~~~~~~~~~~~~~~~~~~~~~~~~~~~
%~~~~~~~~~~~~~~~~~~~~~~~~~~~~~~~~~~~~~~~~~~~~~~~~~~~

\section{Congruence surfaces from trace fields of higher degree}\label{section:higherdegcomp}

\subsection{Reduction to a small list of fields}\label{section:compred}
In the case where $[K:\Q] > 2$, an analysis like the one in Section \ref{sec:quadfield} is extremely difficult as the splitting behavior of primes becomes much more complicated in arbitrary number fields.
We must therefore instead rely on computational results regarding $\zeta_K(-1)$ and constraints on the splitting behavior of $2$ and $3$ to rule out many of the remaining fields.
Notice that the techniques of Section \ref{sec:quadfield}, specifically Proposition \ref{quadprop}, generalize immediately to give the following proposition for number fields of arbitrary degree.

\begin{proposition} \label{Prop}
If $\A$ is a quaternion algebra over a totally real number field $K$ giving rise to a congruence surface of genus $212$ and $\frac{|\zeta_K(-1)|}{2^{[K:\Q]-1}}=\frac{211}{n}$ for $n\in\{2,3,4,6\}$, then one of the following holds:
\begin{itemize}
\item $\frac{|\zeta_K(-1)|}{2^{[K:\Q]-1}} = \frac{211}{6}$ and either
$$\Ram_f(\A) = \{\p, \q_1, \dots, \q_{s}\},$$
where $\Nr(\p) = 7$, $\Nr(\q_i) = 2$, or
$$\Ram_f(\A) = \{\p_1, \p_2,\q_1, \dots, \q_{s}\},$$
where $\Nr(\p_1) = 4$, $\Nr(\p_2) = 3$, $\Nr(\q_i) = 2$.
\item $\frac{|\zeta_K(-1)|}{2^{[K:\Q] - 1}} = \frac{211}{4}$ and
$$\Ram_f(\A) = \{\p_1, \p_2,\q_1, \dots, \q_{s}\},$$
where $\Nr(\p_1) = \Nr(\p_2) = 3$, $\Nr(\q_i) = 2$.
\item $\frac{|\zeta_K(-1)|}{2^ {[K:\Q] - 1}} = \frac{211}{3}$ and
$$\Ram_f(\A) = \{\p, \q_1, \dots, \q_{s}\},$$
where $\Nr(\p) = 4$, $\Nr(\q_i) = 2$.
\item $\frac{|\zeta_K(-1)|}{2^{K:\Q] - 1}} = \frac{211}{2}$ then
$$\Ram_f(\A) = \{\p, \q_1, \dots, \q_{s}\},$$
where $\Nr(\p) = 3$, $\Nr(\q_i) = 2$.
\end{itemize}
By convention, we allow for the possibility that $s=0$ by which we mean that there are no primes of the form $\q_i$ in the $\Ram_f(\A)$ above.
\end{proposition}
\noindent Notice in particular that we have excluded the case of $\frac{|\zeta_K(-1)|}{2^{[K:\Q] - 1}} = \frac{211}{4}$ and
$$\Ram_f(\A) = \{\p, \q_1, \dots, \q_{s}\},$$
where $\Nr(\p)=5$ and $\Nr(\q_i) = 2$, as such a ramification set will never eliminate $3$-torsion by Lemma \ref{lem:3split}.

A more thorough analysis of the above ramification sets shows that some of these primes must ramify in $K$ in order to eliminate torsion in $P(\A^1)$.
We can therefore use Theorem \ref{ram} to give conditions on the discriminant $\Delta_K$, for the ramification sets indicated above to be realized. 
First note in all of these cases, none of the primes $\p \nmid 2$ split in $K(i)/K$ by Lemma \ref{lem:2split} and thus Lemma \ref{lem:2ram} combined with Theorem \ref{ram} show that $4 \mid \Delta_K$ is a necessary condition for any of the ramification sets above to produce an algebra $\A$ such that $P(\A^1)$ has no torsion.
Moreover, in the case where $\frac{|\zeta_K(-1)|}{2} \in\{ \frac{211}{4}, \frac{211}{2}\}$, Lemmas \ref{lem:3split} and \ref{lem:3ram} give that both $2$ and $3$ must ramify in $K$ and hence $12 \mid \Delta_K$ is a necessary condition on the trace field $K$.
Table \ref{table:badzetawithdiv} in the appendix provides a list of all fields from Table \ref{table:badzeta} where these further divisibility conditions on $\Delta_K$ hold.

In the case where $\frac{|\zeta_K(-1)|}{2} \in\{ \frac{211}{3}, \frac{211}{6}\}$, the ramification sets listed in Proposition \ref{Prop} require some $\p \in \Ram_f(\A)$ lying over 2, which necessarily wildly ramifies by Lemma \ref{lem:2ram}. 
Moreover, in each case it could hold that $\Nr(\p) = 2$, and since wild ramification only gives a lower bound on the power of $\p$ dividing $\mathcal{D}_K$ in Theorem \ref{ram}, no improvement can be made on the condition that $4 \mid \Delta_K$.
As a result, in the final section of this paper we analyze the fields in Table \ref{table:badzetawithdiv} computationally to show that the ramification sets in Proposition \ref{Prop} either define algebras with torsion, or simply do not define algebras,  in each individual field.

%~~~~~~~~~~~~~~~~~~~~~~~~~~~~~~~~~~~~~~~~~~~~~~~~~~~
%~~~~~~~~~~~~~~~~~~~~~~~~~~~~~~~~~~~~~~~~~~~~~~~~~~~
%~~~~~~~~~~~~~~~~~~~~~~~~~~~~~~~~~~~~~~~~~~~~~~~~~~~

\subsection{Computations in the remaining fields}\label{section:comptables}
To conclude, we present computations of the splitting behavior of $2$ and $3$ in the fields provided in Table 3. 
None of the methods from the preceding two sections can circumvent the need to check this computationally, as the possibility of wild ramification in these fields allows for no tighter restrictions on $K$ and specifically $\Delta_K$. 
Using SAGE \cite{Sage}, we begin by analyzing the fields where $\frac{|\zeta_K(-1)|}{2^{[K:\Q] - 1}} = \frac{211}{3}$

\begin{proposition}\label{211/3}
For the fields in Table 3, where $\frac{|\zeta_K(-1)|}{2^{[K:\Q] - 1}} = \frac{211}{3}$, 2 factors as follows.
\begin{table}[H]
\centering
\begin{tabular}{|| c | c | c | c ||} \hline
$\Delta_K$ & Minimal polynomial & Factorization & Norms \\
\hline
$​13396$ & $x^3-x^2-25x+29$ & $\p^3$ & $\Nr(\p) = 2$ \\
\hline
$1471216$ & $x^5-2x^4-7x^3+6x^2+8x-4$ & $\p_1\p_2^3$ & $\Nr(\p_1) = 4, \Nr(\p_2) = 2$ \\
\hline
$1630076$ & $x^5-2x^4-9x^3+17x^2+4x-12$ & $\p_1\p_2^2$ & $\Nr(\p_1) = 8, \Nr(\p_2) = 2$ \\
\hline
$1723364$ & $x^5-2x^4-7x^3+13x^2+8x-11$ & $\p_1\p_2^3$ & $\Nr(\p_1) = 4, \Nr(\p_2) = 2$ \\
\hline
$17386832$ & $x^6-3x^5-4x^4+10x^3+6x^2-4x-2$ & $\p_1\p_2^5$ & $\Nr(\p_1) = \Nr(\p_2) = 2$ \\
\hline
$22340432$ & $x^6-x^5-8x^4+5x^3+16x^2-5x-1$  & $\p^3$ & $\Nr(\p) = 4$ \\
\hline
$23556176$ & $x^6-x^5-8x^4+5x^3+16x^2-7x-7$ & $\p^3$ &  $\Nr(\p) = 4$ \\
\hline
\end{tabular}
\end{table}
\end{proposition}
From Proposition \ref{Prop} we see that only the four fields with a prime $\p \mid 2$ and $\Nr(\p) = 4$ could possibly give rise to a genus 212 congruence surface from a maximal order.
However, none of these fields have a prime $\p\mid 2$ with even ramification index, so by Lemma \ref{lem:2ram} none of the primes over $2$ split in $K(i)/K$.
Therefore if $\A$ is a quaternion algebra over these fields with ramification set as in Proposition \ref{Prop}, then $P(\A^1)$ necessarily has 2-torsion.
Consequently, there is no construction of a genus 212 congruence surface constructed from a maximal order in a quaternion algebra over any of the fields in Proposition \ref{211/3}.
\begin{proposition}\label{211/6}
For the fields in Table 3, where $\frac{|\zeta_K(-1)|}{2^{[K:\Q] - 1}} = \frac{211}{6}$, the factorization of $2$ is the following
\begin{table}[H]
\centering
\begin{tabular}{|| c | c | c | c ||} \hline
$\Delta_K$ & Minimal polynomial & Factorization  & Norms \\
\hline
$1060708$ & $x^5-2x^4-7x^3+13x^2+10x-17$ & $\p_1\p_2^3$ & $\Nr(\p_1) = 4, \Nr(\p_2) = 2$ \\
\hline
$12694016$ & $x^6-8x^4-2x^3+16x^2+8x-1$ & $\p_1^2\p_2^2$ & $\Nr(\p_1) =  4, \Nr(\p_2) = 2$\\
\hline
$15004240$ & $x^6-2x^5-11x^4+16x^3+35x^2-26x-17$  & $\p^3$ & $\Nr(\p) = 4$ \\
\hline
$15378496$ & $x^6-2x^5-5x^4+8x^3+6x^2-6x-1$ & $\p^2$ & $\Nr(\p) = 8$ \\
\hline
%$17386832$ & $x^6-3x^5-4x^4+10x^3+6x^2-4x-2$ & $\p_1\p_2^5$ & $\Nr(\p_1) = \Nr(\p_2) = 2$ \\
%\hline
$154050496$ & $x^7-x^6-8x^5+6x^4+13x^3-9x^2-x+1$& $\p_1\p_2^2$ & $\Nr(\p_1) = 2, \Nr(\p_2) = 8$ \\
\hline
\end{tabular}
\end{table}
\noindent and the factorization of $3$ is the following
\begin{table}[H]
\begin{tabular}{|| c | c | c | c ||} \hline
$\Delta_K$ & Minimal polynomial & Factorization & Norms \\
\hline
$1060708$ & $x^5-2x^4-7x^3+13x^2+10x-17$ & $3$ & $\Nr(3) = 729$\\
\hline
$12694016$ & $x^6-8x^4-2x^3+16x^2+8x-1$ & $3$ & $\Nr(3) = 729$\\
\hline
$15004240$ & $x^6-2x^5-11x^4+16x^3+35x^2-26x-17$ & $\p_1\p_2$ & $\Nr(\p_1) = \Nr(\p_2) = 27$  \\
\hline
$15378496$ & $x^6-2x^5-5x^4+8x^3+6x^2-6x-1$ & $\p_1\p_2$ & $\Nr(\p_1) = \Nr(\p_2) = 27$ \\
\hline
%$17386832$ & $x^6-3x^5-4x^4+10x^3+6x^2-4x-2$ & $3$ & $\Nr(3) = 729$ \\
%\hline
$154050496$ & $x^7-x^6-8x^5+6x^4+13x^3-9x^2-x+1$& $\p_1\p_2$ & $\Nr(\p_1) = \Nr(\p_2) = 27$ \\
\hline
\end{tabular}
\end{table}
\end{proposition}
As there are no norm 3 primes in these fields, we see by Proposition \ref{Prop} that in order to have a genus 212 congruence surface from a maximal order we must have a quaternion algebra $\A$ with
$$\Ram_f(\A) = \{\p, \q_1, \dots, \q_{s}\},$$
where $\Nr(\p) = 7$, $\Nr(\q_i) = 2$.
In all but one field, every norm 2 prime has odd ramification index and thus, by Lemma \ref{lem:2ram}, these primes do not split in $K(i)/K$.
Therefore in these fields, a ramification set of the aforementioned form defines an algebra $\A$ such that $P(\A^1)$ has $2$-torsion.
The lone field where there exists a norm 2 prime with even ramification index is the field defined by $p(x) = x^6-8x^4-2x^3+16x^2+8x-1$.
Hence $\Ram_f(\A) = \{\p, \q\}$ where $\Nr(\p) = 7, \Nr(\q) = 2$ since there is only one norm $2$ prime. 
However this is a field of degree six and, by the discussion in Section \ref{sec:hypback}, in order to obtain a hyperbolic surface from $\A$, we must have that $|\Ram_f(\A)|$ is odd.
Therefore there is no such $\A$ with this ramification set and hence there is no construction of a genus $212$ congruence surface constructed from a maximal order with trace field $K$ from the list in Proposition \ref{211/6}.

To conclude we consider the field defined by the polynomial $p(x) = x^4-20x^2+95$, which has $\frac{|\zeta_K(-1)|}{2^{[K:\Q] - 1}} = \frac{211}{30}$.
\begin{proposition}\label{Fin}
Let $K$ be the number field with minimal polynomial $p(x) = x^4 - 20x^2 + 95$.
If a congruence surface from a maximal order of genus $g = 212$ has trace field $K$, then there is a quaternion algebra $\A$ over $K$ with one of the following ramification sets
\begin{enumerate}
\item $\Ram_f(\A) = \{\p, \q_1, \dots, \q_s\}$, where $\Nr(\p) = 31$ and $\Nr(\q_i) = 2$,
\item $\Ram_f(\A) = \{\p_1, \p_2, \q_1, \dots, \q_s\}$, where $\Nr(\p_1) = 16$, $\Nr(\p_2) = 3$, and $\Nr(\q_i) = 2$,
\item $\Ram_f(\A) = \{\p_1, \p_2, \q_1, \dots, \q_s\}$, where $\Nr(\p_1) = 11$, $\Nr(\p_2) = 4$, and $\Nr(\q_i) = 2$.
\end{enumerate}
Again by convention, we allow for the possibility that $s=0$ by which we mean that there are no primes of the form $\q_i$ in $\Ram_f(\A)$.
\end{proposition}
\begin{proof}
In $K$, $2$ factors as $\Fr{r}^2$ where $\Nr(\Fr{r}) = 4$. 
Thus in order for the first ramification set in Proposition \ref{Fin} to be achieved, we must have $\Ram_f(\A) = \{\p\}$ where $\Nr(\p) = 31$.
Since $31 \equiv 3 \pmod 4$, by Lemma \ref{lem:2ram} this ramification set defines an algebra $\A$ such that $P(\A^1)$ has 2-torsion.
For the other two ramification sets we must have $\Ram_f(\A) = \{\p_1, \p_2\}$ since there are no norm 2 primes in $\mathcal{O}_K$.
This is impossible, as $[K:\Q] = 4$, so we must have that $|\Ram_f(\A)|$ is odd.
Hence there is no such algebra with the second or third ramification set in Proposition \ref{Fin}.
\end{proof}

%~~~~~~~~~~~~~~~~~~~~~~~~~~~~~~~~~~~~~~~~~~~~~~~~~~~
%~~~~~~~~~~~~~~~~~~~~~~~~~~~~~~~~~~~~~~~~~~~~~~~~~~~
%~~~~~~~~~~~~~~~~~~~~~~~~~~~~~~~~~~~~~~~~~~~~~~~~~~~

\section{The proof of Theorem \ref{thm:main}}

To conclude we recap our results from the previous sections to prove Theorem \ref{thm:main}, which we begin by restating.

\begin{mthm}
Contingent on the validity of Dokchitser's algorithm to numerically compute L-functions, there is no closed congruence surface constructed from a maximal order of genus 212.
\end{mthm}

\begin{proof}
Indeed if such a surface was to exist then Equation \eqref{eq:3} gives that
$$211 = \frac{|\zeta_K(-1)|}{2^{[K:\Q] - 1}}\prod_{\mathfrak{p} \in \Ram_f(\A)}(\Nr(\mathfrak{p}) - 1),$$
for some totally real number field $K$ and some quaternion algebra $\A$ over $K$ such that $P(\A^1)$ contains no torsion.
For the above equation to be satisfied, we must have that $\frac{|\zeta_K(-1)|}{2^{[K:\Q] - 1}}=\frac{s}{n}$ where $s\in\{1,211\}$ and $n\in\N$.
Therefore Proposition \ref{prop:degupperbnd} gives that $[K:\Q]\le 10$.
Proposition \ref{prop:noQ} and Theorem \ref{thm:noquad} show that $K$ cannot be $\Q$ or a real quadratic, so we reduce to the case that $3\le [K:\Q]\le 10$.
Moreover, the beginning of the proof of Proposition \ref{quadprop} actually shows that $n=1$ is impossible in general, an argument which we briefly repeat.

If $n=1$ then we have that 
$$211 = s\prod_{\mathfrak{p} \in \Ram_f(\A)}(\Nr(\mathfrak{p}) - 1),$$
where again $s\in\{1,211\}$.
If $s=1$, then there is some prime $\p\in\Ram_f(\A)$ with $\Nr(\p)=212$ which is clearly absurd as the latter is not a prime power.
If $s=211$, then every $\p\in\Ram_f(\A)$ has the property that $\Nr(\p)=2$.
However Lemma \ref{lem:3split} then shows that $P(\A^1)$ contains $3$-torsion.
We are therefore reduced to analyzing fields for which $n\neq 1$.

For such fields, Equation \eqref{eq:4} and Proposition \ref{TAK} show that there exists numbers $U_d$ depending only on $d=[K:\Q]$ such that if $U_d < \Delta_K$ then
$$\frac{211}{2} < \frac{|\zeta_K(-1)|}{2^{d-1}}.$$
Using Voight's enumeration of all totally real number fields with discriminant less than $U_d$ \cite{Voight}, we then find all number fields where $\Delta_K < U_d$ in each degree.
SAGE's zeta function command \cite{Sage} then computes the value of $\frac{|\zeta_K(-1)|}{2^{d-1}}$ for each field, as in Corollary \ref{cor:zetalist}\footnote{ The full list of values for all such number fields is far too long to include in the paper, however we collect this list on the second author's website which can be found \textcolor{blue}{\href{http://pages.iu.edu/~nimimill/zetatables.pdf}{here}}. If this author's web address changes, this list will continue to be hosted by the author wherever he moves to.}.
From these computations, we list in Table \ref{table:badzeta} all number fields where $ \frac{|\zeta_K(-1)|}{2^{[K:\Q] - 1}} = \frac{211}{n}$ for some $n \in \N$, or where $ \frac{|\zeta_K(-1)|}{2^{[K:\Q] - 1}} = \frac{1}{n}$ and there exists some $k \mid n$ such that $211 k + 1$ is a prime power.
In the latter case, Theorem \ref{1/n} and Table \ref{table:badzeta} shows that no such congruence surface exists (notice that $2\nmid\Delta_K$ for the fields in question).
In the former case, the results in Section \ref{section:compred} show that no such construction exists for all fields in Table \ref{table:badzeta}, except for those which we list in Table \ref{table:badzetawithdiv}. 
To conclude, the results in Section \ref{section:comptables} prove no such construction exists for all fields listed in Table \ref{table:badzetawithdiv}. 
Consequently there is no congruence surface from a maximal order of genus $212$, for any of the number fields listed in Table \ref{table:badzeta}.
As these are the only possible field with $3\le[K:\Q]\le 10$ for which a construction is possible, this completes the proof.
\end{proof}

%~~~~~~~~~~~~~~~~~~~~~~~~~~~~~~~~~~~~~~~~~~~~~~~~~~~
%~~~~~~~~~~~~~~~~~~~~~~~~~~~~~~~~~~~~~~~~~~~~~~~~~~~
%~~~~~~~~~~~~~~~~~~~~~~~~~~~~~~~~~~~~~~~~~~~~~~~~~~~

\newpage

%~~~~~~~~~~~~~~~~~~~~~~~~~~~~~~~~~~~~~~~~~~~~~~~~~~~
%~~~~~~~~~~~~~~~~~~~~~~~~~~~~~~~~~~~~~~~~~~~~~~~~~~~
%~~~~~~~~~~~~~~~~~~~~~~~~~~~~~~~~~~~~~~~~~~~~~~~~~~~

\section{Appendix}
\begin{table}[H]
\centering
\begin{tabular}{||c | c |  c||}\hline
$\Delta_K$ & $p(x)$ & $\frac{|\zeta_K(-1)|}{2^{[K:\Q] -1}}$ \\
\hline
5 & $x^2 - 5$ & 1/60  \\
\hline
49 &  $x^3- x^2- 2x+1$ & 1/84 \\
\hline
7825 & $x^3-25x-45$ & 211/6 \\
\hline
9812 & $x^3-x^2-19x+33$ & 211/4 \\
\hline
13396 & $x^3-x^2-25x+29$ & 211/3 \\
\hline
725 & $x^4-x^3-3x^2+x+1$ & 1/60 \\
\hline
1125 & $x^4-x^3-4x^2+4x+1$ & 1/30 \\
\hline
2225 & $x^4-x^3-5x^2+2x+4$ & 1/10 \\
\hline
38000 & $x^4-20x^2+95$ & 211/30 \\
\hline
148889 & $x^4-2x^3-15x^2+16x+51$ & 211/3 \\
\hline
150057 & $x^4-2x^3-7x^2+5x+7$ & 211/3 \\
\hline
1060708 & $x^5-2x4-7x^3+13x^2+10x-17$ & 211/6 \\
\hline
1459417 & $x^5-2x4-7x^3+11x^2+9x-2$ & 211/3 \\
\hline
1471216 & $x^5-2x^4-7x^3+6x^2+8x-4$ & 211/3 \\
\hline
1630076 & $x^5-2x^4-9x^3+17x^2+4x-12$ & 211/3 \\
\hline
1723364 & $x^5-2x^4-7x^3+13x^2+8x-11$ & 211/3 \\
\hline
12694016 & $x^6-8x^4-2x^3+16x^2+8x-1$ & 211/6 \\ 
\hline
15004240 & $x^6-2x^5-11x^4+16x^3+35x^2-26x-17$ & 211/6 \\
\hline
15378496 & $x^6-2x^5-5x^4+8x^3+6x^2-6x-1$ & 211/6 \\
\hline
15700473 & $x^6-12x^4-2x^3+39x^2+12x-19$ & 211/6 \\
\hline
17386832 & $x^6-3x^5-4x^4+10x^3+6x^2-4x-2$ & 211/3 \\
\hline
17801408 & $x^6-9x^4-4x^3+16x^2+14x+3$ & 211/4 \\
\hline
18967381 & $x^6-2x^5-8x^4+11x^3+20x^2-14x-17$ & 211/4 \\
\hline
22340432 & $x^6-x^5-8x^4+5x^3+16x^2-5x-1$ & 211/3 \\
\hline
23556176 & $x^6-x^5-8x^4+5x^3+16x^2-7x-7$ & 211/3 \\
\hline
26768537 & $x^6-x^5-11x^4+11x^2+x-2$  & 211/2 \\
\hline  
154050496 & $x^7-x^6-8x^5+6x^4+13x^3-9x^2-x+1$ & 211/6 \\
\hline
225111553 & $x^7-2x^6-6x^5+9x^4+12x^3-8x^2-8x-1$ & 211/4 \\
\hline
236583241 & $x^7-x^6-9x^5+4x^4+16x^3-5x^2-5x+1$ & 211/3 \\
\hline
343318749 & $x^7-2x^6-6x^5+12x^4+8x^3-17x^2+3$ & 211/2 \\
\hline
\end{tabular}
\caption{Number Fields with Relevant $\frac{|\zeta_K(-1)|}{2^{[K:\Q] - 1}}$ Values}
\label{table:badzeta}
\end{table}
\newpage
%% \begin{table}[H]
%% \centering
%% \begin{tabular}{||c | c | c ||} \hline
%% $[K:\Q]$ & $\Delta_K$ & $\Delta_K^{1/[K:\Q]}$  \\
%% \hline
%% 3                      & 21922                    & 27.99                                                  \\
%% \hline
%% 4                      & 254161                   & 22.45                                                  \\
%% \hline
%% 5                      & 2946779                  & 19.67                                                  \\
%% \hline
%% 6                      & 34165459                 & 18.01                                                  \\
%% \hline
%% 7                      & 396120284                & 16.91                                                  \\
%% \hline
%% 8                      & 4592687600               & 16.13                                                  \\
%% \hline
%% 9                      & 53248420466              & 15.55                                                  \\
%% \hline
%% 10                     & 617371467322             & 15.10                                                  \\
%% \hline
%% 11                     & 7157912391165            & 14.74      \\                        
%% \hline                    
%% \end{tabular}
%% \caption{Upper Bounds for Discriminant in Fixed Degree}
%% \end{table}
\begin{table}[H]
\centering
\begin{tabular}{||c | c |  c||}\hline
$\Delta_K$ & $p(x)$ & $\frac{|\zeta_K(-1)|}{2^{[K:\Q] -1}}$ \\
\hline
13396 & $x^3-x^2-25x+29$ & 211/3 \\
\hline
38000 & $x^4-20x^2+95$ & 211/30 \\
\hline
1060708 & $x^5-2x4-7x^3+13x^2+10x-17$ & 211/6 \\
\hline
1471216 & $x^5-2x^4-7x^3+6x^2+8x-4$ & 211/3 \\
\hline
1630076 & $x^5-2x^4-9x^3+17x^2+4x-12$ & 211/3 \\
\hline
1723364 & $x^5-2x^4-7x^3+13x^2+8x-11$ & 211/3 \\
\hline
12694016 & $x^6-8x^4-2x^3+16x^2+8x-1$ & 211/6 \\ 
\hline
15004240 & $x^6-2x^5-11x^4+16x^3+35x^2-26x-17$ & 211/6 \\
\hline
15378496 & $x^6-2x^5-5x^4+8x^3+6x^2-6x-1$ & 211/6 \\
\hline
17386832 & $x^6-3x^5-4x^4+10x^3+6x^2-4x-2$ & 211/3 \\
\hline
22340432 & $x^6-x^5-8x^4+5x^3+16x^2-5x-1$ & 211/3 \\
\hline
23556176 & $x^6-x^5-8x^4+5x^3+16x^2-7x-7$ & 211/3 \\
\hline
154050496 & $x^7-x^6-8x^5+6x^4+13x^3-9x^2-x+1$ & 211/6 \\
\hline
\end{tabular}
\caption{Remaining Number Fields where Splitting Behavior of 2,3 Must be Checked Computationally}
\label{table:badzetawithdiv}
\end{table}

\newpage

\begin{table}[H]
\begin{tabular}{ccccccccc}
\begin{tabular}{| c | c |} \hline
d & $\frac{\zeta_K(-1)}{2}$ \\
\hline
2 & 1/24 \\
\hline
3 & 1/12 \\
\hline
5 & 1/60  \\
\hline
6 & 1/4 \\
\hline
7 & 1/3 \\
\hline
10 & 7/12 \\
\hline
11 & 7/12 \\
\hline
13 & 1/12 \\
\hline
14 & 5/6 \\
\hline
15 & 1 \\
\hline
17 & 1/6 \\
\hline
19 & 19/12 \\
\hline
21 & 1/6 \\
\hline
22 & 23/12 \\
\hline
23 & 5/3 \\
\hline
26 & 25/12 \\
\hline
29 & 1/4 \\
\hline
30 & 17/6 \\
\hline
31 & 10/3 \\
\hline
33 & 1/2 \\
\hline
34 & 23/6 \\
\hline
35 & 19/6 \\
\hline
37 & 5/12 \\
\hline
38 & 41/12 \\
\hline
39 & 13/3 \\
\hline
41 & 2/3 \\
\hline
42 & 9/2 \\
\hline
43 & 21/4 \\
\hline
46 & 37/6 \\
\hline
47 & 14/3 \\
\hline
51 & 13/2 \\
\hline
53 & 7/12 \\
\hline
55 & 23/3 \\
\hline
57 & 7/6 \\
\hline
58 & 33/4 \\
\hline
59 & 85/12 \\
\hline
61 & 11/12 \\
\hline
62 & 7 \\
\hline
65 & 4/3 \\
\hline
66 & 28/3 \\
\hline
67 & 41/4 \\
\hline
69 & 1 \\
\hline
70 & 67/6 \\
\hline
71 & 29/3 \\
\hline
73 & 11/6 \\
\hline
74 & 41/4 \\
\hline
\end{tabular}
&

&
\begin{tabular}{| c | c |} \hline
d & $\frac{\zeta_K(-1)}{2}$ \\
\hline
77 & 1 \\
\hline
78 & 23/2 \\
\hline
79 & 14 \\
\hline
82 & 27/2 \\
\hline
83 & 43/4 \\
\hline
85 & 3/2 \\
\hline
86 & 155/12 \\
\hline
87 & 13 \\
\hline
89 & 13/6 \\
\hline
91 & 103/6 \\
\hline
93 & 3/2 \\
\hline
94 & 53/3 \\
\hline
95 & 43/3 \\
\hline
97 & 17/6 \\
\hline
101 & 19/12 \\
\hline
102 & 103/6 \\
\hline
103 & 19 \\
\hline
105 & 3 \\
\hline
106 & 87/4 \\
\hline
107 & 197/12 \\
\hline
109 & 9/4 \\
\hline
110 & 103/6 \\
\hline
111 & 61/3 \\
\hline
113 & 3 \\
\hline
114 & 22 \\
\hline
115 & 139/6 \\
\hline
118 & 277/12 \\
\hline
119 & 62/3 \\
\hline
122 & 77/4 \\
\hline
123 & 45/2 \\
\hline
127 & 80/3 \\
\hline
129 & 25/6 \\
\hline
130 & 173/6 \\
\hline
131 & 93/4 \\
\hline
133 & 17/6 \\
\hline
134 & 301/12 \\
\hline
137 & 4 \\
\hline
138 & 77/3 \\
\hline
139 & 127/4 \\
\hline
141 & 3 \\
\hline
142 & 63/2 \\
\hline
143 & 73/3 \\
\hline
145 & 16/3 \\
\hline
149 & 35/12 \\
\hline
157 & 43/12 \\
\hline
161 & 16/3 \\
\hline
\end{tabular}
&

&
\begin{tabular}{| c | c |} \hline
d & $\frac{\zeta_K(-1)}{2}$ \\
\hline
165 & 11/3 \\
\hline
173 & 13/4 \\
\hline
177 & 13/2 \\
\hline
178 & 128/3 \\
\hline
179 & 157/4 \\
\hline
181 & 19/4 \\
\hline
185 & 19/3 \\
\hline
193 & 49/6 \\
\hline
197 & 49/12 \\
\hline
201 & 49/6 \\
\hline
205 & 17/3 \\
\hline
209 & 47/6 \\
\hline
213 & 5 \\
\hline
217 & 29/3 \\
\hline
221 & 16/3 \\
\hline
229 & 27/4 \\
\hline
233 & 53/6 \\
\hline
237 & 35/6 \\
\hline
241 & 71/6 \\
\hline
249 & 23/2 \\
\hline
253 & 15/2 \\
\hline
257 & 10 \\
\hline
265 & 40/3 \\
\hline
269 & 83/12 \\
\hline
273 & 37/3 \\
\hline
277 & 103/12 \\
\hline
281 & 25/2 \\
\hline
285 & 8 \\
\hline
293 & 85/12 \\
\hline
301 & 31/3 \\
\hline
305 & 41/3 \\
\hline
309 & 59/6 \\
\hline
313 & 50/3 \\
\hline
317 & 101/12 \\
\hline
321 & 33/2 \\
\hline
329 & 47/3 \\
\hline
337 & 19 \\
\hline
341 & 59/6 \\
\hline
345 & 55/3 \\
\hline
349 & 151/12 \\
\hline
353 & 16 \\
\hline
357 & 11 \\
\hline
365 & 65/6 \\
\hline
373 & 161/12 \\
\hline
377 & 53/3 \\
\hline
381 & 77/6 \\
\hline
\end{tabular}
&

&
\begin{tabular}{| c | c |} \hline
d & $\frac{\zeta_K(-1)}{2}$ \\
\hline
385 & 71/3 \\
\hline
389 & 151/12 \\
\hline
393 & 43/2 \\
\hline
397 & 57/4 \\
\hline
401 & 43/2 \\
\hline
409 & 79/3 \\
\hline
413 & 73/6 \\
\hline
417 & 47/2 \\
\hline
421 & 209/12 \\
\hline
429 & 16 \\
\hline
433 & 163/6 \\
\hline
437 & 77/6 \\
\hline
445 & 55/3 \\
\hline
449 & 51/2 \\
\hline
453 & 91/6 \\
\hline
457 & 30 \\
\hline
461 & 61/4 \\
\hline
465 & 28 \\
\hline
469 & 20 \\
\hline
473 & 51/2 \\
\hline
481 & 101/3 \\
\hline
485 & 101/6 \\
\hline
489 & 187/6 \\
\hline
493 & 119/6 \\
\hline
497 & 27 \\
\hline
498 & 560/3 \\
\hline
501 & 20 \\
\hline
505 & 36 \\
\hline
509 & 215/12 \\
\hline
517 & 127/6 \\
\hline
521 & 31 \\
\hline
533 & 37/2 \\
\hline
537 & 67/2 \\
\hline
541 & 301/12 \\
\hline
545 & 95/3 \\
\hline
553 & 119/3 \\
\hline
554 & 2525/12 \\
\hline
557 & 233/12 \\
\hline
561 & 118/3 \\
\hline
565 & 151/6 \\
\hline
569 & 109/3 \\
\hline
573 & 22 \\
\hline
\end{tabular}
\end{tabular}
\caption{Quadratic Fields $\Q(\sqrt{d})$, and $\frac{\zeta_K(-1)}{2^{[K:\Q] -1}}$  for $\Delta_K < 574$}
\label{table:quadfield}
\end{table}

\bibliographystyle{abbrv}
\bibliography{Biblio}

\end{document}